\documentclass{article}

\usepackage{fullpage,color}
\usepackage{amsmath,amssymb,amsthm}
\usepackage{authblk}
\usepackage{tikz,subcaption}

\DeclareMathOperator*{\EE}{\mathbb{E}}
\DeclareMathOperator*{\VV}{\mathbb{V}}
\DeclareMathOperator{\Inf}{Inf}
\DeclareMathOperator{\im}{im}
\DeclareMathOperator{\cf}{cf}
\providecommand{\RR}{\mathbb{R}}
\providecommand{\cB}{\mathcal{B}}
\providecommand{\cS}{\mathcal{S}}
\providecommand{\cY}{\mathcal{Y}}
\providecommand{\pH}{\mathcal{H}}
\providecommand{\ppH}{\pH'}

\newif\ifcomments\commentsfalse
\ifcomments
\providecommand{\todo}[1]{{\color{red} #1}}
\else
\providecommand{\todo}[1]{}
\fi

\newtheorem{theorem}{Theorem}[section]
\newtheorem{proposition}[theorem]{Proposition}
\newtheorem{lemma}[theorem]{Lemma}
\theoremstyle{definition}
\newtheorem{definition}{Definition}[section]

\title{Orthogonal basis for functions over a slice of the Boolean hypercube}
\author{Yuval Filmus}
\affil{Technion --- Israel Institute of Technology, Haifa, Israel}

\begin{document}
\maketitle

\begin{abstract}
We present a simple, explicit orthogonal basis of eigenvectors for the Johnson and Kneser graphs, based on Young's orthogonal representation of the symmetric group.
Our basis can also be viewed as an orthogonal basis for the vector space of all functions over a slice of the Boolean hypercube (a set of the form $\{(x_1,\ldots,x_n) \in \{0,1\}^n : \sum_i x_i = k\}$), which refines the eigenspaces of the Johnson association scheme; our basis is orthogonal with repsect to \emph{any} exchangeable measure.
More concretely, our basis is an orthogonal basis for all multilinear polynomials $\RR^n \to \RR$ which are annihilated by the differential operator $\sum_i \partial/\partial x_i$. As an application of the last point of view, we show how to lift low-degree functions from a slice to the entire Boolean hypercube while maintaining properties such as expectation, variance and $L^2$-norm.

As an application of our basis, we streamline Wimmer's proof of Friedgut's theorem for the slice.
Friedgut's theorem, a fundamental result in the analysis of Boolean functions, states that a Boolean function on the Boolean hypercube with low total influence can be approximated by a Boolean junta (a function depending on a small number of coordinates). Wimmer generalized this result to slices of the Boolean hypercube, working mostly over the symmetric group, and utilizing properties of Young's orthogonal representation. Using our basis, we show how the entire argument can be carried out directly on the slice.
\end{abstract}

\section{Introduction} \label{sec:introduction}

Functions over the Boolean hypercube $\{0,1\}^n$ are often studied using the tools of \emph{Fourier analysis} (see O'Donnell's excellent recent monograph~\cite{ODonnell}). The crucial idea is to study functions from the point of view of the \emph{Fourier basis}, an orthonormal basis of functions over the Boolean hypercube. In this work, we consider functions on a different domain, a \emph{slice of the Boolean hypercube} $\binom{[n]}{k} = \{(x_1,\ldots,x_n) \in \{0,1\}^n : \sum_i x_i = k\}$; we always assume that $k \leq n/2$. Such functions arise naturally in coding theory, in the context of constant-weight codes, and have recently started appearing in theoretical computer science as well. In this work we provide an explicit orthogonal basis for the vector space of functions on a slice.

The slice has been studied in algebraic combinatorics under the name \emph{Johnson association scheme}, and in spectral graph theory in relation to the Johnson and Kneser graphs. Our basis is the analog of the Fourier basis for the scheme, and it refines the decomposition induced by the primitive idempotents.
Our basis is also an orthogonal basis for the eigenvectors of the Johnson and Kneser graphs, and any other graph belonging to the Bose--Mesner algebra of the Johnson association scheme. Such (weighted) graphs arise in Lov\'asz's proof of the Erd\H{o}s--Ko--Rado theorem~\cite{Lovasz}, and in Wilson's proof~\cite{Wilson} of a $t$-intersecting version of the theorem.

Despite the name, it is perhaps best to view the slice $\binom{[n]}{k}$ as the set of cosets of $S_k \times S_{n-k}$ inside $S_n$. This point of view suggests ``lifting'' an orthogonal basis from the symmetric group to the slice. Following Bannai and~Ito~\cite{BannaiIto}, the relevant representations of the symmetric group are those corresponding to partitions $(n-d)+d$ for $d \leq k$. Our basis arises from Young's orthogonal representation of the symmetric group. However, we present the basis and prove its properties without reference to the symmetric group at all.
One feature that is inherited from the symmetric group is the lack of a canonical basis: our basis relies on the ordering of the coordinates. 

Dunkl~\cite{Dunkl76} showed that the space of functions over the slice $\binom{[n]}{k}$ can be identified with the space of multilinear polynomials in $n$ variables $x_1,\ldots,x_n$ which are annihilated by the differential operator $\sum_{i=1}^n \partial/\partial x_i$; the input variable $x_i$ is an indicator variable for the event that $i$ belongs to the input set. Functions annihilated by this operator were termed \emph{harmonic} by Dunkl~\cite{Dunkl79}. Our basis forms an orthogonal basis for the space of harmonic multilinear polynomials for \emph{every} exchangeable measure (a measure invariant under the action of $S_n$). As a consequence, we show how to lift a low-degree function from the slice $\binom{[n]}{k}$ to the Boolean cube (under an appropriate measure) while maintaining some of its properties such as expectation, variance and $L^2$-norm.

\medskip

Wimmer~\cite{Wimmer} recently generalized a fundamental theorem of Friedgut~\cite{FriedgutJunta} from the Boolean hypercube to the slice. Friedgut's theorem, sometimes known as Friedgut's junta theorem, states that a Boolean function on the Boolean hypercube with low total influence is close to a Boolean junta (a function depending on a small number of variables). Although Wimmer's main theorem is a statement about functions on the slice, Wimmer lifts the given function to the symmetric group, where most of his argument takes place, exploiting essential properties of Young's orthogonal representation. Eventually, a hypercontractive property of the slice (due to Lee~and~Yau~\cite{LeeYau}) is invoked to complete the proof. As an application of our basis, we give a streamlined version of Wimmer's proof in which our basis replaces the appeal to the symmetric group and to Young's orthogonal representation.

\paragraph{Note added in proof} Since writing this paper, we have learned that the same basis has been constructed by Srinivasan~\cite{Srinivasan} in a beautiful paper. Srinivasan in fact constructs an extended basis for the entire Boolean cube $\{0,1\}^n$, which he identifies with the canonical Gelfand--Tsetlin basis~\cite{VO}, and shows that it is orthogonal with respect to all exchangeable measures. However, he provides neither an explicit description of the basis elements, nor even a canonical indexing scheme for the basis elements. Instead, he gives a recursive algorithm that constructs the basis. We believe that both approaches have merit.

\paragraph{Related work} Apart from Friedgut's theorem, several other classical results in Fourier analysis of Boolean functions have recently been generalized to the slice. O'Donnell and Wimmer~\cite{OW1,OW2} generalized the Kahn--Kalai--Linial theorem~\cite{KKL} to the slice, and deduced a robust version of the Kruskal--Katona theorem. Filmus~\cite{Filmus} generalized the Friedgut--Kalai--Naor theorem~\cite{FKN} to the slice.

Filmus, Kindler, Mossel and Wimmer~\cite{FKMW} and Filmus and Mossel~\cite{FM} generalized the invariance principle~\cite{MOO} to the slice. The invariance principle on the slice compares the behavior of low-degree harmonic multilinear polynomials on a slice $\binom{[n]}{k}$ and on the Boolean hypercube $\{0,1\}^n$ with respect to the corresponding product measure $\mu_{k/n}$. If the harmonic multilinear polynomial $f$ has degree $d$ and unit variance, then the invariance principle states that for any Lipschitz functional $\varphi$, \[
 |\EE_\sigma[\varphi(f)] - \EE_{\mu_p}[\varphi(f)]| = \tilde{O}\left(\sqrt{\frac{d}{\sqrt{p(1-p)n}}}\right),
\]
where $\sigma$ is the uniform distribution on the slice. The invariance principle can be used to lift results such as the Kindler--Safra theorem~\cite{Kindler,KindlerSafra} and Majority is Stablest from the Boolean hypercube to the slice.

Filmus and Mossel also give basis-free proofs for some of the results appearing in this paper. For example, they give a basis-free proof for the fact that the $L^2$-norm of a low-degree harmonic multilinear polynomial on the slice is similar to its $L^2$-norm on the Boolean cube under the corresponding product measure.

\paragraph{Synopsis} We describe the space of harmonic multilinear polynomials in Section~\ref{sec:polynomials}. Our basis is defined in Section~\ref{sec:basis}, in which we also compute the norms of the basis elements. We show that our basis forms a basis for functions on the slice in Section~\ref{sec:slice}, in which we also show how to lift low-degree functions from the slice to the entire hypercube, and explain why our basis is an orthogonal basis of eigenvectors for the Johnson and Kneser graphs. Section~\ref{sec:influences} and Section~\ref{sec:wimmer-friedgut} are devoted to the proof of the Wimmer--Friedgut theorem.

\paragraph{Notation} We use the notation $[n] = \{1,\ldots,n\}$. The cardinality of a set $S$ is denoted $|S|$. If $S \subseteq [n]$ and $\pi \in S_n$ (the symmetric group on $[n]$) then $S^\pi = \{ \pi(x) : x \in S \}$. We use the same notation in other similar circumstances. We compose permutations from right to left, so $\beta\alpha$ means apply $\alpha$ then $\beta$. We use the falling power notation: $n^{\underline{k}} = n(n-1)\cdots(n-k+1)$ (the number of terms is $k$). For example, $\binom{n}{k} = n^{\underline{k}}/k!$. A function is \emph{Boolean} if its values are in $\{0,1\}$.

\paragraph{Acknowledgements} The author thanks Karl Wimmer for helpful discussions and encouragement, and Qing Xiang and Rafael Plaza for pointing out Srinivasan's paper~\cite{Srinivasan}. The paper was written while the author was a member of the Institute for Advanced Study at Princeton, NJ.

We thank Bruno Loff for clarifying (in 2022!) the proof of Theorem~\ref{thm:orthogonal}.

This material is based upon work supported by the National Science Foundation under agreement No.~DMS-1128155. Any opinions, findings and conclusions or recommendations expressed in this material are those of the authors, and do not necessarily reflect the views of the National Science Foundation.

\section{Harmonic multilinear polynomials} \label{sec:polynomials}


We construct our basis as a basis for the vector space of \emph{harmonic multilinear polynomials} over $x_1,\ldots,x_n$, a notion defined below. For simplicity, we only consider polynomials over $\RR$, but the framework works just as well over any field of characteristic zero.

\begin{definition} \label{def:harmonic}
A polynomial $P \in \RR[x_1,\ldots,x_n]$ is \emph{multilinear} if all monomials of $P$ are squarefree (not divisible by any $x_i^2$).

A multilinear polynomial $P \in \RR[x_1,\ldots,x_n]$ is \emph{harmonic} if
\[ \sum_{i=1}^n \frac{\partial P}{\partial x_i} = 0. \]
We denote the vector space of harmonic multilinear polynomials over $x_1,\ldots,x_n$ by $\pH_n$.

The \emph{degree} of a non-zero multilinear polynomial is the maximal number of variables in any monomial.
We denote the subspace of $\pH_n$ consisting of polynomials of degree at most $d$ by $\pH_{n,d}$.

A polynomial has \emph{pure degree $d$} if all its monomials have degree $d$.
We denote the subspace of $\pH_n$ consisting of polynomials of degree exactly $d$ by $\ppH_{n,d}$.
\end{definition}

\todo{This definition of a harmonic function comes from the field of coinvariant spaces. A general function $f$ is \emph{harmonic} if for each $k \geq 1$,
\[ \sum_{i=1}^n \frac{\partial^k f}{\partial x_i^k} = 0. \]
In the case of multilinear polynomials, it is enough to consider $k = 1$.
We suspect that our basis can be extended to an orthogonal basis of general harmonic polynomials.}

The following lemma calculates the dimension of the vector space of harmonic multilinear polynomials of given degree.

\begin{lemma} \label{lem:dimension}
 All polynomials in $\pH_n$ have degree at most $n/2$. For $d \leq n/2$,
\[ \dim \pH_{n,d} = \binom{n}{d}, \quad \dim \ppH_{n,d} = \binom{n}{d} - \binom{n}{d-1}, \]
 where $\binom{n}{-1} = 0$.
\end{lemma}
\begin{proof}
 We start by proving the upper bound on the degree of polynomials in $\pH_n$. Let $P \in \pH_n$ have degree $\deg P = d$. The pure degree $d$ part of $P$ is also in $H_n$, and so we can assume without loss of generality that $P$ has pure degree $d$. For any $\mathbf{y} = y_1,\ldots,y_n$, the univariate polynomial $P(t\mathbf{1} + \mathbf{y})$ (where $\mathbf{1}$ is the constant vector) doesn't depend on $t$, since
\[ \frac{dP(\mathbf{y} + t\mathbf{1})}{dt} = \sum_{i=1}^n \frac{\partial P}{\partial x_i}(\mathbf{y} + t\mathbf{1}) = 0. \]
In particular, if $M$ is any monomial in $P$ with coefficient $\alpha \neq 0$ and $\mathbf{y}$ is the vector with $y_i = -1$ whenever $x_i$ appears in $M$ and $y_i = 0$ otherwise, then $P(\mathbf{v} + \mathbf{1}) = P(\mathbf{v}) = (-1)^d \alpha \neq 0$, showing that $P$ must contain some monomial supported on variables not appearing in $M$, since $y_i = 1$ only for $x_i$ not appearing in $M$. In particular, $2d \leq n$.

 We proceed with the formula for $\ppH_{n,d}$; the formula for $\pH_{n,d}$ easily follows. When $d = 0$, the formula clearly holds, so assume $d \geq 1$.  The vector space of all multilinear polynomial of pure degree $d$ over $x_1,\ldots,x_n$ has dimension $\binom{n}{d}$. Denote by $\cf(P,M)$ the coefficient of the monomial $M$ in $P$. Harmonicity is the set of conditions
\[ \sum_{\substack{i \in [n]: \\ x_i \notin M}} \cf(P,x_iM) = 0 \text{ for all monomials $M$ of degree $d-1$}. \]
There are $\binom{n}{d-1}$ conditions, showing that $\dim \ppH_{n,d} \geq \binom{n}{d} - \binom{n}{d-1}$. In order to prove equality, we need to show that the conditions are linearly independent. We do this by showing that there is a polynomial $P$ having pure degree $d$ satisfying all but one of them, that is
\[ \sum_{i=1}^n \frac{\partial P}{\partial x_i} = x_1 \cdots x_{d-1}. \]
Such a polynomial is given by
\[ P = \frac{1}{d} \sum_{t=1}^d (-1)^{t+1} \binom{d}{t} \EE_{\substack{A \in [d-1]\colon |A| = d-t \\ B \in [n] \setminus [d-1]\colon |B| = t}} x_A x_B, \]
using the notation $x_S = \prod_{i \in S} x_i$. Indeed,
\begin{align*}
\sum_{i=1}^n \frac{\partial P}{\partial x_i} &= x_1\cdots x_{d-1} + \frac{1}{d} \sum_{t=2}^d \left[t \cdot (-1)^{t+1} \binom{d}{t} + (d-t+1) \cdot (-1)^t \binom{d}{t-1}\right] \EE_{\substack{A \in [d-1]\colon |A| = d-t \\ B \in [n] \setminus [d-1]\colon |B| = t-1}} x_A x_B \\ &= x_1 \cdots x_{d-1}. \qedhere
\end{align*}
\end{proof}

Frankl and Graham~\cite{FranklGraham} gave a basis for $\pH_n$.

\begin{definition} \label{def:frankl-graham}
For $d \leq n/2$, a \emph{sequence of length $d$} is a sequence $S = s_1,\ldots,s_d$ of distinct numbers in $[n]$. The set of all sequence of length $d$ is denoted by $\cS_{n,d}$, and the set of all sequences is denoted by $\cS_n$.

For any two disjoint sequences $A,B \in \cS_{n,d}$ we define
\[ \chi_{A,B} = \prod_{i=1}^d (x_{a_i} - x_{b_i}). \]
\end{definition}

The basis functions will be $\chi_{A,B}$ for appropriate $A,B$.

\begin{definition} \label{def:top-set}
For $d \leq n/2$, let $A,B \in \cS_{n,d}$ be disjoint. We say that \emph{$A$ is smaller than $B$}, written $A < B$, if $a_i < b_i$ for all $i \in [d]$. Similarly, we say that \emph{$A$ is at most $B$}, written $A \leq B$, if $a_i \leq b_i$ for all $i \in [d]$.

A sequence $B \in \cS_n$ is a \emph{top set} if $B$ is increasing and for some disjoint sequence $A$ of the same length, $A < B$. The set of top sets of length $d$ is denoted by $\cB_{n,d}$, and the set of all top sets is denoted by $\cB_n$.
\end{definition}

The following lemma is mentioned without proof in~\cite{FranklGraham}.

\begin{lemma} \label{lem:top-sets}
For $0 \leq d \leq n/2$, $|\cB_{n,d}| = \binom{n}{d} - \binom{n}{d-1}$, where $\binom{n}{-1} = 0$.
\end{lemma}
\begin{proof}
 We encode each sequence $B \in \cS_{n,d}$ as a $\pm 1$ sequence $\beta_0,\dots,\beta_d$ as follows. We put $\beta_0 = 1$, and for $i \in [d]$, $\beta_i = 1$ if $i \notin B$ and $\beta_i = -1$ if $i \in B$. It is not hard to check that $B$ is a top set iff all running sums of $\beta$ are positive. Each sequence $\beta$ is composed of $d$ entries $-1$ and $n-d+1$ entries $1$. The probability that such a sequence has all running sums positive is given by the solution to Bertrand's ballot problem: it is $\frac{(n-d+1)-d}{n+1}$. Therefore the total number of top sets is
\[ \frac{(n-d+1)-d}{n+1} \binom{n+1}{d} = \frac{n-d+1}{n+1} \binom{n+1}{n-d+1} - \frac{d}{n+1} \binom{n+1}{d} = \binom{n}{n-d} - \binom{n}{d-1}. \qedhere \]
\end{proof}

We can now give Frankl and Graham's basis, which is given in~\cite{FranklGraham} without proof.

\begin{lemma} \label{lem:frankl-graham}
For each $B \in \cB_n$, let $\phi(B) \in \cS_{n,|B|}$ be any sequence satisfying $\phi(B) < B$. The set $\{ \chi_{\phi(B),B} : B \in \cB_n \}$ is a basis for $\pH_n$. Moreover, for $d \leq n/2$, the set $\{ \chi_{\phi(B),B} : B \in \cB_{n,d} \}$ is a basis for $\ppH_{n,d}$.
\end{lemma}
\begin{proof}
It is clearly enough to prove that $X_d = \{ \chi_{\phi(B),B} : B \in \cB_{n,d} \}$ is a basis for $\ppH_{n,d}$. In view of Lemma~\ref{lem:dimension} and Lemma~\ref{lem:top-sets}, it is enough to prove that all functions in $X_d$ belong to $\ppH_{n,d}$, and that $X_d$ is linearly independent. Clearly all functions in $X_d$ are multilinear polynomial of pure degree $d$. To show that they are harmonic, notice that if $A < B$, where $|A| = |B| = d$, then
\[
\sum_{i=1}^n \frac{\partial \chi_{A,B}}{\partial x_i} = \sum_{j=1}^d \left(\frac{\partial \chi_{A,B}}{\partial x_{a_j}} + \frac{\partial \chi_{A,B}}{\partial x_{b_j}}\right) = \sum_{j=1}^d (1-1) \frac{\chi_{A,B}}{x_{a_j} - x_{b_j}} = 0.
\]

It remains to prove that $X_d$ is linearly independent. For an increasing sequence $S \in \cS_{n,d}$, let $\Pi(S)$ be the monomial $\Pi(S) = \prod_{i=1}^d x_{s_i}$. If $\Pi(S)$ appears in $\chi_{\phi(B),B}$ then $S \leq B$. Consider now the matrix representing $X_d$ in the basis $\{\Pi(S) : S \in \cS_{n,d}\}$ arranged in an order compatible with the partial order of $\cS_{n,d}$. The resulting matrix is in echelon form and so has full rank, showing that $X_d$ is a linearly independent set.
\end{proof}

We comment that $\dim \ppH_{n,d} = \binom{n}{d} - \binom{n}{d-1}$ is the dimension of the irreducible representation of $S_n$ corresponding to the partition $(n-d) + d$, as easily calculated using the hook formula. This is not a coincidence: indeed, the set of standard Young tableaux of shape $(n-d),d$ is in bijection with $\cB_{n,d}$ by identifying the top sets with the contents of the second row of each such tableau (the tableau can be completed uniquely).

\section{Young's orthogonal basis} \label{sec:basis}

In this section we will construct an orthogonal basis for $\pH_n$, and calculate the norms of the basis elements. Our basis will be orthogonal with respect to a wide class of measures.

\begin{definition} \label{def:invariant}
 A probability distribution over the variables $x_1,\ldots,x_n$ is \emph{exchangeable} if it is invariant under permutations of the indices. Given an exchangeable distribution $\mu$, we define an inner product on $\pH_n$ by
\[ \langle f,g \rangle = \EE_\mu[fg]. \]
 The norm of $f \in \pH_n$ is $\|f\| = \sqrt{\langle f,f \rangle}$.
\end{definition}

We are now ready to define the basis.

\begin{definition} \label{def:young-basis}
For $B \in \cB_{n,d}$, define
\[
\chi_B = \sum_{\substack{A \in \cS_{n,d}\colon \\ A < B}} \chi_{A,B}.
\]
For $d \leq n/2$, we define
\[
\chi_d = \chi_{2,4,\ldots,2d} = \chi_{1,3,\ldots,2d-1;2,4,\ldots,2d} = \prod_{i=1}^d (x_{2i-1} - x_{2i}).
\]

\emph{Young's orthogonal basis} for $\pH_n$ is
\[ \cY_n = \{ \chi_B : B \in \cB_n \}. \]
Young's orthogonal basis for $\ppH_{n,d}$ is
\[ \cY_{n,d} = \{ \chi_B : B \in \cB_{n,d} \}. \]
\end{definition}

We stress that the sequences $A$ in the definition of $\chi_B$ need not be increasing.

The following theorem justifies the name ``orthogonal basis''.

\begin{theorem} \label{thm:orthogonal}
The set $\cY_n$ is an orthogonal basis for $\pH_n$ with respect to any exchangeable measure.
The set $\cY_{n,d}$ is an orthogonal basis for $\ppH_{n,d}$ with respect to any exchangeable mesure.
In particular, the subspaces $\ppH_{n,d}$ for different $d$ are mutually orthogonal.
\end{theorem}
\begin{proof}
Lemma~\ref{lem:frankl-graham} shows that each $\chi_B \in \cY_{n,d}$ lies in $\ppH_{n,d}$. The technique used in proving the lemma shows that $\cY_{n,d}$ is a basis for $\ppH_{n,d}$, but this will also follow from Lemma~\ref{lem:dimension} once we prove that the functions in $\cY_{n,d}$ are pairwise orthogonal. In fact, we will prove the following more general claim: if $B_1,B_2 \in \cB_n$ and $B_1 \neq B_2$ then $\chi_{B_1},\chi_{B_2}$ are orthogonal. This will complete the proof of the theorem.

Consider any $B_1 \in \cB_{n,d_1}$ and $B_2 \in \cB_{n,d_2}$, where $B_1 \neq B_2$. Our goal is to prove that $\langle \chi_{B_1},\chi_{B_2} \rangle = 0$. We have
\[
\langle \chi_{B_1}, \chi_{B_2} \rangle = \sum_{\substack{A_1 < B_1 \\ A_2 < B_2}} \EE[\chi_{A_1,B_1}\chi_{A_2,B_2}].
\]
We call each of the terms $\chi_{A_1,B_1} \chi_{A_2,B_2}$ appearing in this expression a \emph{quadratic product}. We will construct a sign-flipping involution among the quadratic products, completing the proof. The involution is also allowed to have fixed points; in this case, the expectation of the corresponding quadratic product vanishes.

Consider a quadratic product
\[ \chi_{A_1,B_1} \chi_{A_2,B_2} = \prod_{i=1}^{d_1} (x_{a_{1,i}} - x_{b_{1,i}}) \prod_{j=1}^{d_2} (x_{a_{2,j}} - x_{b_{2,j}}). \]
We can represent this quadratic product as a directed graph $G$ on the vertex set $A_1 \cup A_2 \cup B_1 \cup B_2$. For each factor $x_i - x_j$ in the quadratic product, we draw an edge from $i$ to $j$; all edges point in the direction of the larger vertex (the vertex having a larger index). We further annotate each edge with either $1$ or $2$, according to which of $\chi_{A_1,B_1},\chi_{A_2,B_2}$ it corresponds to. Every variable $x_i$ appears in at most two factors, and so the total degree of each vertex is at most $2$. Therefore the graph decomposes as an undirected graph into a disjoint union of paths and cycles. The annotations on the edges alternate on each connected component.

Every directed graph $G'$ in which edges point in the direction of the larger vertex, the total degree of each vertex is at most $2$, and the annotations on the edges alternate in each connected component, is the graph corresponding to some quadratic product $\chi_{A'_1,B'_1} \chi_{A'_2,B'_2}$. The value of $A'_1,B'_1,A'_2,B'_2$ can be read using the annotations on the edges. We define $\EE[G'] = \EE[\chi_{A'_1,B'_1} \chi_{A'_2,B'_2}]$.

Since $B_1 \neq B_2$, some connected component must have a vertex with in-degree $1$. Choose the connected component $C$ satisfying this property having the largest vertex. We construct a sequence of intervals inside $C$, with the property that each of the endpoints $x,y$ of each interval is either an endpoint of $C$, or is connected to the rest of $C$ via a vertex $z>x,y$. Furthermore, each interval, other than possibly the last one, contains some vertex with in-degree $1$. The sequence terminates with an interval containing an odd number of edges.

When $C$ is a path, the first interval $I_0$ is the entire path. When $C$ is a cycle with maximal vertex $M$, the first interval $I_0$ is the path obtained by removing $M$ from $C$. Given an interval $I_t$ with an even number of edges, we can break it into two (possibly empty) subintervals terminating at the maximal point $M_t$ of $I_t$: $I_t = J_t \to M_t \gets K_t$. Note that not both $J_t,K_t$ can be empty since $I_t$ contains some vertex with in-degree $1$. If $J_t$ is empty then we define $I_{t+1} = K_t$, which terminates the sequence. Similarly, if $K_t$ is empty then we define $I_{t+1} = J_t$, which terminates the sequence. If both $J_t,K_t$ are non-empty then at least one of them has a vertex with in-degree $1$. We let $I_{t+1}$ be the sub-interval among $J_t,K_t$ with the larger maximal point.

Since the intervals decrease in size, the sequence eventually terminates at some interval $I_t = v_1,\ldots,v_\ell$ having an odd number of edges (so $\ell$ is even). We now consider two graphs obtained from $G$. The first graph $G^\pi$ is obtained by applying the permutation $\pi$ which maps $v_i$ to $v_{\ell+1-i}$ and fixes all other vertices. The second graph $G^r$ is obtained by detaching $I_t$ from $G$, reversing it, and attaching it back to $G$; see Figure~\ref{fig:orthogonal}.

If we run the same construction on $G^r$ then we get the same connected component $C$ and the same sequence of intervals $I_0,\ldots,I_t$, and so $(G^r)^r = G$, that is, the mapping $G \mapsto G^r$ is an involution.

\begin{figure}

\centering

\subcaptionbox{The graph $G$, with $I_1$ and $I_2$ highlighted}[0.2\textwidth]{
\begin{tikzpicture}
\node[blue] (A) at (-1,0) {$2$};
\node[blue] (B) at (-0.5, 0.866) {$3$};
\node[blue] (C) at (0.5, 0.866) {$4$};
\node[blue] (D) at (1,0) {$1$};
\node (E) at (0.5, -0.866) {$6$};
\node[cyan] (F) at (-0.5, -0.866) {$5$};

\draw[->, blue] (A) -- (B) node[midway,left] {$1$};
\draw[->, blue] (B) -- (C) node[midway,above] {$2$};
\draw[<-, blue] (C) -- (D) node[midway,right] {$1$};
\draw[->] (D) -- (E) node[midway,right] {$2$};
\draw[<-] (E) -- (F) node[midway,below] {$1$};
\draw[<-, cyan] (F) -- (A) node[midway,left] {$2$};
\end{tikzpicture}
}
\qquad
\subcaptionbox{The graph $G^\pi$, where $\pi = (1\;2)(3\;4)$}[0.2\textwidth]{
\begin{tikzpicture}
\node[blue] (A) at (-1,0) {$1$};
\node[blue] (B) at (-0.5, 0.866) {$4$};
\node[blue] (C) at (0.5, 0.866) {$3$};
\node[blue] (D) at (1,0) {$2$};
\node (E) at (0.5, -0.866) {$6$};
\node[cyan] (F) at (-0.5, -0.866) {$5$};

\draw[->, blue] (A) -- (B) node[midway,left] {$1$};
\draw[->, blue] (B) -- (C) node[midway,above] {$2$};
\draw[<-, blue] (C) -- (D) node[midway,right] {$1$};
\draw[->] (D) -- (E) node[midway,right] {$2$};
\draw[<-] (E) -- (F) node[midway,below] {$1$};
\draw[<-, cyan] (F) -- (A) node[midway,left] {$2$};
\end{tikzpicture}
}
\qquad
\subcaptionbox{The graph $G^r$, with reversed edge in purple}[0.2\textwidth]{
\begin{tikzpicture}
\node[blue] (A) at (-1,0) {$1$};
\node[blue] (B) at (-0.5, 0.866) {$4$};
\node[blue] (C) at (0.5, 0.866) {$3$};
\node[blue] (D) at (1,0) {$2$};
\node (E) at (0.5, -0.866) {$6$};
\node[cyan] (F) at (-0.5, -0.866) {$5$};

\draw[->, blue] (A) -- (B) node[midway,left] {$1$};
\draw[<-, magenta] (B) -- (C) node[midway,above] {$2$};
\draw[<-, blue] (C) -- (D) node[midway,right] {$1$};
\draw[->] (D) -- (E) node[midway,right] {$2$};
\draw[<-] (E) -- (F) node[midway,below] {$1$};
\draw[<-, cyan] (F) -- (A) node[midway,left] {$2$};
\end{tikzpicture}
}

\caption{Illustration of the proof of Theorem~\ref{thm:orthogonal} for $\chi_{A_1,B_1} = (x_2 - x_3)(x_1 - x_4)(x_5 - x_6)$ and $\chi_{A_2,B_2} = (x_3 - x_4)(x_2 - x_5)(x_1 - x_6)$}
\label{fig:orthogonal}

\end{figure}

Since the measure is exchangeable, $\EE[G] = \EE[G^\pi]$. The graphs $G^\pi,G^r$ differ only in the direction of some edges: the edge between $v_i$ and $v_{i+1}$ in $G^\pi$ has the same direction as the edge between $v_{\ell+1-i}$ and $v_{\ell-i}$ in $G$, which is the same as its direction in $G^r$. Let $\sigma_\pi(v_i,v_{i+1}) = 1$ if the edge between $v_i$ and $v_{i+1}$ goes from $v_i$ to $v_{i+1}$ in $G^\pi$, let $\sigma_\pi(v_i,v_{i+1}) = -1$ if it goes in the other direction, and define $\sigma_r(v_i,v_{i+1})$ analogously. Then
\[
 \EE[G] = \EE[G^\pi] = \sigma \EE[G^r], \text{ where } \sigma = \prod_{i=1}^{\ell-1} \sigma_\pi(v_i,v_{i+1}) \sigma_r(v_i,v_{i+1}).
\]
Since $\sigma_\pi(v_i,v_{i+1}) = -\sigma_r(v_{\ell-i},v_{\ell-i+1})$, we have
\[
 \sigma = (-1)^{\ell-1} \prod_{i=1}^{\ell-1} \sigma_\pi(v_i,v_{i+1}) \sigma_\pi(v_{\ell-i},v_{\ell-i+1}) = -1,
\]
and so $\EE[G] = -\EE[G^r]$. 

By construction, $G^r$ (but not $G^\pi$) corresponds to some quadratic product $\chi_{A'_1,B_1} \chi_{A'_2,B_2}$. Therefore $\chi_{A_1,B_1} \chi_{A_2,B_2} \mapsto \chi_{A'_1,B_1} \chi_{A'_2,B_2}$ is a sign-flipping involution on the collection of all quadratic products, completing the proof.
\end{proof}

In order to complete the picture, we need to evaluate the norms of the basis elements $\chi_B$, which necessarily depend on the measure.

\begin{theorem} \label{thm:norms}
 Let $B \in \cB_{n,d}$. The squared norm of $\chi_B$ is $\|\chi_B\|^2 = c_B \|\chi_d\|^2$ with respect to any exchangeable measure, where
\[ c_B = \prod_{i=1}^d \frac{(b_i-2(i-1))(b_i-2(i-1)-1)}{2}. \]
\end{theorem}
(Recall that $\chi_d = \chi_{2,4,\ldots,2d}$.)
\begin{proof}
 We consider first the case in which the exchangeable measure is the measure $\nu_p$ for some $p \in [0,1]$. Under this measure, the variables $x_1,\ldots,x_n$ are independent, with $\Pr[x_i = -p] = 1-p$ and $\Pr[x_i = 1-p] = p$. The expectation of each $x_i$ is $\EE[x_i] = (1-p)(-p) + p(1-p) = 0$, while the variance is $\EE[x_i^2] = (1-p)p^2 + p(1-p)^2 = p(1-p)$. The squared norm of $\chi_B$ is
\[
 \|\chi_B\|^2 = \langle \chi_B,\chi_B \rangle = \sum_{\substack{A_1,A_2 \in \cS_d\colon \\ A_1,A_2 < B}} \EE[\chi_{A_1,B} \chi_{A_2,B}].
\]
 In the proof of Theorem~\ref{thm:orthogonal} we associated a directed graph with each \emph{quadratic product} $\chi_{A_1,B} \chi_{A_2,B}$: the vertices are $A_1 \cup A_2 \cup B$, and the edges point from $a_{1,i}$ and $a_{2,i}$ to $b_i$ for each $i \in [d]$, annotated by $1$ or $2$ according to whether they came from $\chi_{A_1,B}$ or from $\chi_{A_2,B}$. Since each vertex appears at most twice, the graph decomposes as a sum of paths and cycles. The edges point from $A_1,A_2$ to $B$, and so each vertex either has in-degree $0$ (if it is in $A_1 \cup A_2$) or in-degree $2$ (if it is in $B$). Therefore the paths and cycles have the following forms, respectively:
\begin{gather*}
 \alpha_1 \to \beta_1 \gets \alpha_2 \to \beta_2 \gets \cdots \gets \alpha_\ell \to \beta_\ell \gets \alpha_{\ell+1}, \\
 \alpha_1 \to \beta_1 \gets \alpha_2 \to \beta_2 \gets \cdots \gets \alpha_\ell \to \beta_\ell \gets \alpha_1.
\end{gather*}
 Here the $\alpha_i$ belong to $A_1 \cup A_2$, and the $\beta_i$ belong to $B$. The corresponding factors of $\chi_{A_1,B} \chi_{A_2,B}$ are, respectively:
\begin{gather*}
 (x_{\alpha_1} - x_{\beta_1}) (x_{\alpha_2} - x_{\beta_1}) (x_{\alpha_2} - x_{\beta_2}) \cdots (x_{\alpha_\ell} - x_{\beta_\ell}) (x_{\alpha_{\ell+1}} - x_{\beta_\ell}), \\
 (x_{\alpha_1} - x_{\beta_1}) (x_{\alpha_2} - x_{\beta_1}) (x_{\alpha_2} - x_{\beta_2}) \cdots (x_{\alpha_\ell} - x_{\beta_\ell}) (x_{\alpha_1} - x_{\beta_\ell}).
\end{gather*}
 We proceed to calculate the expectation of each of these factors under $\nu_p$. The expectation of a monomial is zero unless each variable appears exactly twice, in which case the expectation is $(p(1-p))^\ell$ (since each monomial has total degree $2\ell$). In the case of a path, there is exactly one such monomial, namely $x_{\beta_1}^2 x_{\beta_2}^2 \cdots x_{\beta_\ell}^2$. In the case of a cycle, there are two such monomials: $x_{\beta_1}^2 x_{\beta_2}^2 \cdots x_{\beta_\ell}^2$ and $x_{\alpha_1}^2 x_{\alpha_2}^2 \cdots x_{\alpha_\ell}^2$. Both monomials appear with unit coefficient. Notice that $\ell$ is the size of the subset of $B$ appearing in the path or cycle. Hence the expectation of the entire quadratic product is $2^C (p(1-p))^d$, where $C = C(A_1,A_2)$ is the number of cycles. In total, we get
\[ \|\chi_B\|_{\nu_p}^2 = \sum_{\substack{A_1,A_2 \in \cS_d\colon \\ A_1,A_2 < B}} 2^{C(A_1,A_2)} (p(1-p))^d. \]
 We proceed to show that
\[ \sum_{\substack{A_1,A_2 \in \cS_d\colon \\ A_1,A_2 < B}} 2^{C(A_1,A_2)} = 2^d c_B = \prod_{i=1}^d (b_i-2(i-1))(b_i-2(i-1)-1). \]
 The quantity on the right enumerates the sequences $\alpha_{1,1},\alpha_{2,1},\alpha_{1,2},\alpha_{2,2},\dots,\alpha_{1,d},\alpha_{2,d} \in \cS_{n,2d}$ in which $\alpha_{1,i},\alpha_{2,i} \leq b_i$ for all $i \in [d]$, which we call \emph{legal sequences}. We show how to map legal sequences into quadratic products in such a way that $\chi_{A_1,B} \chi_{A_2,B}$ has exactly $2^{C(A_1,A_2)}$ preimages.

 Let $\alpha_{1,1},\alpha_{2,1},\alpha_{1,2},\alpha_{2,2},\dots,\alpha_{1,d},\alpha_{2,d} \in \cS_{n,2d}$ be a legal sequence. We construct a quadratic product $\chi_{A_1,B} \chi_{A_2,B}$ alongside its associated directed graph. We maintain the following invariant: after having processed $\alpha_{1,i},\alpha_{2,i}$ (and so adding $b_i$ to the graph), all vertices belonging to cycles have appeared in the sequence, and out of each path, exactly one vertex (from $B$) has not appeared previously in the sequence. Furthermore, the endpoints of each path do not belong to $B$ (and so have incoming edges) and have different annotations.

 We start with the empty product and graph. At step $i$ we process $\alpha_{1,i},\alpha_{2,i}$. Suppose first that $\alpha_{1,i},\alpha_{2,i} \neq b_i$. If $\alpha_{1,i} \notin B$, then we set $a_{1,i} = \alpha_i$ and add the edge $\alpha_i \xrightarrow{1} b_i$. If $\alpha_{1,i} \in B$ then necessarily $\alpha_{1,i} < b_i$, and so it appears in some component $C$. We locate the endpoint $x$ whose incoming edge is labelled $2$, set $a_{1,i} = x$ and add the edge $x \xrightarrow{1} b_i$. Note that the other endpoint of $C$ is labelled $1$. We do the same for $\alpha_{2,i}$. It is routine to check that we have maintained the invariant.

 Suppose next that one of $\alpha_{1,i},\alpha_{2,i}$ equals $b_i$, say $\alpha_{2,i} = b_i$. We process $\alpha_{1,i}$ as in the preceding step. Let $x$ be the other endpoint of the path containing $b_i$. We set $a_{2,i} = x$ and connect $x$ to $b_i$, completing the cycle. This completes the description of the mapping.

 We proceed to describe the multivalued inverse mapping, from a quadratic product to a legal sequence. We process the quadratic product in $d$ steps, updating the graph by removing each vertex mentioned in the legal sequence. We maintain the invariant that each original path remains a path, and each original cycle either remains a cycle or disappears after processing the largest vertex. Furthermore, each edge still points at the larger vertex, the annotations alternate, a vertex $b_i$ not yet processed has two incoming edges, and other vertices have no incoming edges.

 At step $i$, we process $b_i$. Let $x_1,x_2$ be the neighbors of $b_i$ labelled $1,2$, respectively. Suppose first that $x_1 \neq x_2$. We put $\alpha_{1,i} = x_1$ and $\alpha_{2,i} = x_2$, and remove the vertices $x_1,x_2$. The other neighbors of $x_1,x_2$, if any, are connected to $b_i$ with edges pointing away from $b_i$ with annotations $2,1$, respectively. These neighbors had incoming edges and so are $b_j,b_k$ for $j,k>i$. It follows that the invariant is maintained. The case $x_1 = x_2$ corresponds to a cycle whose largest vertex is $b_i$. We either put $\alpha_{1,i} = x_1$ and $\alpha_{2,i} = b_i$ or $\alpha_{1,i} = b_i$ and $\alpha_{2,i} = x_2$, deleting the entire cycle in both cases.

 It is routine to check that the two mappings we have described are inverses. Furthermore, the multivalued mapping from quadratic products to legal sequences has valency $2^{C(A_1,A_2)}$ when processing $\chi_{A,B_1} \chi_{A,B_2}$. This completes the proof of the formula for $2^d c_B$.

 Having considered the measure $\nu_p$, we consider a related measure $\mu_p$. Under this measure the $x_i$ are independent, $\Pr[x_i = 0] = 1-p$, and $\Pr[x_1 = 1] = p$. Note that $(x_1-p,\dots,x_n-p) \sim \nu_p$. Since $(x_i - p) - (x_j - p) = x_i - x_j$, we conclude that
\[ \|\chi_B\|_{\mu_p}^2 = \|\chi_B\|_{\nu_p}^2 = (2p(1-p))^d c_B. \]

 Consider now a general exchangeable measure $m$. Exchangeability implies that for some integers $\gamma_0,\ldots,\gamma_d$,
\[ \|\chi_B\|_m^2 = \sum_{k=0}^d \gamma_k \EE_m\left[\prod_{i=1}^{d-k} x_i^2 \prod_{i=1}^{2k} x_{d-k+i}\right], \]
 since all monomials in $\EE[\chi_B^2]$ have total degree $2d$. Substituting the measure $\mu_p$, which satisfies $\EE[x_i] = \EE[x_i^2] = p$, we obtain
\[ \sum_{k=0}^d \gamma_k p^{d+k} = \|\chi_B\|_{\mu_p}^2 = (2p(1-p))^d c_B. \]
 Reading off the coefficient of $\gamma_k$, we deduce
\[ \gamma_k = (-1)^k \binom{d}{k} 2^d c_B, \]
 and so
\[ \|\chi_B\|_m^2 = 2^d c_B N_d, \quad N_d = \sum_{k=0}^d (-1)^k \binom{d}{k} \EE_m\left[\prod_{i=1}^{d-k} x_i^2 \prod_{i=1}^{2k} x_{d-k+i}\right]. \]
 In order to evaluate $N_d$, we consider $B = 2,4,\ldots,2d$. In this case, $b_i = 2i$ and so
\[ c_B = \prod_{i=1}^d \frac{(b_i-2(i-1))(b_i-2(i-1)-1)}{2} = 1. \]
 We conclude that $2^d N_d = \|\chi_d\|^2$, and the theorem follows.
\end{proof}

Having verified that $\cY_n$ is a basis for $\pH_n$, we can define the corresponding expansion.

\begin{definition} \label{def:expansion}
 Let $f \in \pH_n$. The \emph{Young--Fourier expansion} of $f$ is the unique representation
\[ f = \sum_{B \in \cB_n} \hat{f}(B) \chi_B. \]
\end{definition}

The following simple lemma gives standard properties of this expansion.

\begin{lemma} \label{lem:moments}
 Let $f \in \pH_n$. The following hold with respect to any exchangeable measure. For each $B \in \cB_n$, we have $\hat{f}(B) = \langle f,\chi_B \rangle/\|\chi_B\|^2$. The mean, variance and L2 norm of $f$ are given by
\[
 \EE[f] = \hat{f}(\emptyset), \quad \VV[f] = \sum_{\substack{B \in \cB_n\colon \\ B \neq \emptyset}} \hat{f}(B)^2 c_B \|\chi_{|B|}\|^2, \quad \EE[f^2] = \sum_{B \in \cB_n} \hat{f}(B)^2 c_B \|\chi_{|B|}\|^2,
\]
 where $\emptyset$ is the empty sequence.
\end{lemma}

In particular, if $f \in \pH_n$ then for every exchangeable measure $\pi$ we can write $\EE[f^2] = \sum_d \|\chi_d\|^2 \mathbf{W}^d[f]$, where $\mathbf{W}^d[f] = \sum_{|B|=d} c_B \hat{f}(B)^2$ depends only on $f$. Filmus and Mossel~\cite{FM} give a basis-free proof of this important fact.

The familiar Fourier basis for the Boolean hypercube gives a simple criterion for a function to depend on a variable. The matching criterion in our case is also simple but not as powerful.

\begin{lemma} \label{lem:averaging}
 Let $f \in \pH_n$, and for $m \in [n]$ let $g = \EE_{\pi \in S_m}[f^\pi]$ be the result of averaging $f$ over all permutations of the first $m$ coordinates. Then
\[ g = \sum_{\substack{B \in \cB_n\colon \\ B \cap [m] = \emptyset}} \hat{f}(B) \chi_B. \]
 In particular, $f$ is invariant under permutation of the first $m$ coordinates iff $\hat{f}(B) = 0$ whenever $B$ intersects $[m]$.
\end{lemma}
\begin{proof}
 It is enough to prove the formula for basis functions $f \in \cY_n$. If $f = \chi_B$ where $B \cap [m] = \emptyset$ then $f^\pi = f$ for all $\pi \in S_m$ since $A < B$ iff $A^\pi < B$, and so $g = f$. Suppose next that $f = \chi_B$ where $B$ intersects $[m]$, say $b_j \in [m]$. For $a < b_j$ not belonging to $B$, define
\[ \chi_{a,B} = \sum_{\substack{A \in \cS_{n,|B|}\colon \\ A < B, a_j = a}} \chi_{A,B}. \]
 Since $\chi_B = \sum_{a \in [b_j] \setminus B} \chi_{a,B}$, it is enough to show that $\EE_{\pi \in S_m}[\chi_{a,B}^\pi]$ vanishes. Since the mapping $\pi \mapsto (a\;b_j) \pi$ is an involution on $S_m$, it is enough to notice that $\EE[\chi_{a,B}^{(a\;b_j) \pi}] = -\EE[\chi_{a,B}^\pi]$.
\end{proof}

\section{Slices of the Boolean hypercube} \label{sec:slice}

Harmonic multilinear polynomials appear naturally in the context of \emph{slices of the Boolean hypercube}.

\begin{definition} \label{def:slice}
Let $n$ be an integer and let $k \leq n/2$ be an integer. The \emph{$(n,k)$ slice of the Boolean hypercube} is
\[ \binom{[n]}{k} = \{(x_1,\ldots,x_n) \in \{0,1\}^n : \sum_{i=1}^n x_i = k\}. \]
We also identify $\binom{[n]}{k}$ with subsets of $[n]$ of cardinality $k$. We endow the slice $\binom{[n]}{k}$ with the uniform measure, which is clearly exchangeable.

A \emph{function over the slice} is a function $f\colon \binom{[n]}{k} \to \RR$. Every function $f\colon \RR^n \to \RR$ can be interpreted as a function on the slice in the natural way.
\end{definition}

We proceed to show that $\cY_{n,0} \cup \cdots \cup \cY_{n,k}$ is an orthogonal basis for the slice $\binom{[n]}{k}$.

\begin{theorem} \label{thm:slice-basis}
Let $n$ and $k \leq n/2$ be integers, and put $p = k/n$. The set $\{ \chi_B : B \in \cB_{n,d} \text{ for some } d \leq k \}$ is an orthogonal basis for the vector space of functions on the slice, and for $B \in \cB_{n,d}$,
\[ \|\chi_B\|^2 = c_B 2^d \frac{k^{\underline{d}} (n-k)^{\underline{d}}}{n^{\underline{2d}}} = c_B (2p(1-p))^d \left(1 \pm O_p\left(\frac{d^2}{n}\right)\right). \]
\end{theorem}
\begin{proof}
We prove the formula and estimate for $\|\chi_B\|^2$ below. When $d \leq k$, the formula shows that $\|\chi_B\|^2 \neq 0$ and so $\chi_B \neq 0$ as a function on the slice. Hence the set $\{ \chi_B : B \in \cB_{n,d} \text{ for some } d \leq k \}$ consists of non-zero mutually orthogonal vectors. Lemma~\ref{lem:top-sets} shows that the number of vectors in this set is $\binom{n}{k}$, matching the dimension of the vector space of functions on the slice. Hence this set constitutes a basis for the vector space of functions on the slice.

In order to prove the formula for $\|\chi_B\|^2$, we need to compute
\[ \|\chi_d\|^2 = \EE\left[\prod_{i=1}^d (x_{2i-1} - x_{2i})^2 \right]. \]
The quantity $\prod_{i=1}^d (x_{2i-1} - x_{2i})^2$ is non-zero for a subset $S \in \binom{[n]}{k}$ if $S$ contains exactly one out of each pair $\{x_{2i-1},x_{2i}\}$, in which case the quantity has value $1$. Therefore
\begin{align*}
\|\chi_d\|^2 &= \Pr_{S \in \binom{[n]}{k}} [|S \cap \{x_{2i-1},x_{2i}\}| = 1 \text{ for all } i \in [d]] \\ &= 2^d \Pr_{S \in \binom{[n]}{k}} [x_1,x_3,\ldots,x_{2d-1} \in S \text{ and } x_2,x_4,\ldots,x_{2d} \notin S] \\ &= 2^d \frac{\binom{n-2d}{k-d}}{\binom{n}{k}} = 2^d \frac{(n-2d)!k!(n-k)!}{(k-d)!(n-k-d)!n!} = 2^d \frac{k^{\underline{d}} (n-k)^{\underline{d}}}{n^{\underline{2d}}}.
\end{align*}
This yields the formula for $\|\chi_B\|^2$. We can estimate this expression as follows:
\begin{align*}
2^d \frac{k^{\underline{d}} (n-k)^{\underline{d}}}{n^{\underline{d}}} &= (2p(1-p))^d \frac{\left(1-\frac{1}{k}\right)\left(1-\frac{2}{k}\right)\cdots\left(1-\frac{d-1}{k}\right) \cdot \left(1-\frac{1}{n-k}\right) \left(1-\frac{2}{n-k}\right) \cdots \left(1-\frac{d-1}{n-k}\right)}{\left(1-\frac{1}{n}\right) \left(1-\frac{2}{n}\right) \cdots \left(1-\frac{2d-1}{n}\right)} \\ &= (2p(1-p))^d \frac{\left(1 - \frac{O(d^2)}{k}\right) \left(1 - \frac{O(d^2)}{n-k}\right)}{\left(1 - \frac{O(d^2)}{n}\right)},
\end{align*}
implying the stated estimate.
\end{proof}

The theorem shows that every function on $\binom{[n]}{k}$ can be represented uniquely as a harmonic multilinear polynomial of degree at most $k$ (see Filmus and Mossel~\cite{FM} for a basis-free proof).
The quantity $(2p(1-p))^d$ is the squared norm of $\chi_d$ under the $\mu_p$ measure (defined below) over the entire Boolean hypercube, and this allows us to lift low-degree functions from the slice to the entire Boolean hypercube while preserving properties such as the expectation and variance.

\begin{theorem} \label{thm:slice-cube}
Let $n$ and $k \leq n/2$ be integers, and put $p = k/n$. Let $f\colon \binom{[n]}{k} \to \RR$ be a function whose representation as a harmonic multilinear polynomial has degree $d$, and define $\tilde{f}\colon \{0,1\}^{[n]} \to \RR$ by interpreting this polynomial over $\{0,1\}^{[n]}$.

If we endow $\{0,1\}^{[n]}$ with the measure $\mu_p$ given by $\mu_p(x_1,\ldots,x_n) = p^{\sum_{i=1}^n x_i} (1-p)^{\sum_{i=1}^n (1-x_i)}$ then
\[ \EE[\tilde{f}] = \EE[f], \qquad \|\tilde{f}\|^2 = \|f\|^2 \left(1 \pm O_p\left(\frac{d^2}{n}\right)\right), \qquad \VV[\tilde{f}] = \VV[f] \left(1 \pm O_p\left(\frac{d^2}{n}\right)\right). \]
\end{theorem}
\begin{proof}
The definition of $\tilde{f}$ directly implies that for all $B \in \cB_n$, $\hat{\tilde{f}}(B) = \hat{f}(B)$, and furthermore $\hat{f}(B) = 0$ for $|B| > d$.
According to Lemma~\ref{lem:moments}, $\EE[\tilde{f}] = \hat{\tilde{f}}(\emptyset) = \hat{f}(\emptyset) = \EE[f]$. In order to prove the estimate on the norms, we compute the squared norm of $\chi_\ell$ with respect to $\mu_p$:
\[ \|\chi_\ell\|^2_{\mu_p} = \EE_{\mu_p}\left[\prod_{i=1}^\ell (x_{2i-1}-x_{2i})^2\right] = (2p(1-p))^\ell. \]
Denoting by $\sigma$ the uniform measure on $\binom{[n]}{k}$, Lemma~\ref{lem:moments} and Theorem~\ref{thm:slice-basis} imply that
\begin{align*}
 \|\tilde{f}\|^2 &= \sum_{B \in \cB_n} \hat{\tilde{f}}(B) \|\chi_{|B|}\|^2_{\mu_p} =
 \sum_{B \in \cB_n} \hat{f}(B) (2p(1-p))^{|B|}, \\
 \|f\|^2 &= \sum_{B \in \cB_n} \hat{f}(B) \|\chi_{|B|}\|^2_{\sigma} =
 \sum_{B \in \cB_n} \hat{f}(B) (2p(1-p))^{|B|} \left(1 \pm O_p\left(\frac{d^2}{n}\right)\right),
\end{align*}
implying the estimate on the norms.
The estimate on the variance follows analogously.
\end{proof}

\paragraph{Johnson association scheme}
The Johnson association scheme is an association scheme whose underlying set is $\binom{[n]}{k}$. Instead of describing the scheme itself, we equivalently describe its Bose--Mesner algebra.

\begin{definition} \label{def:johnson}
Let $n,k$ be integers such that $k \leq n/2$.
A square matrix $M$ indexed by $\binom{[n]}{k}$ belongs to the \emph{Bose--Mesner algebra of the $(n,k)$ Johnson association scheme} if $M_{S,T}$ depends only on $|S \cap T|$.
\end{definition}

While it is not immediately obvious, the Bose--Mesner algebra is indeed an algebra of matrices, that is, it is closed under multiplication. Furthermore, it is a commutative algebra, and so all matrices have common eigenspaces. In particular, the algebra is spanned by a basis of primitive idempotents $J_0,\ldots,J_k$. As the following lemma shows, these idempotents correspond to the bases $\cY_{n,0},\ldots,\cY_{n,k}$.

\begin{lemma} \label{lem:idempotents}
Let $n,k$ be integers such that $k \leq n/2$, and let $M$ belong to the Bose--Mesner algebra of the $(n,k)$ Johnson association scheme. The bases $\cY_{n,0},\ldots,\cY_{n,k}$ span the eigenspaces of $M$.
\end{lemma}
\begin{proof}
Since all matrices in the Bose--Mesner algebra have the same eigenspaces, it is enough to consider a particular matrix in the algebra which has $k+1$ distinct eigenvalues. Let $M$ be the matrix corresponding to the linear operator
\[ M\colon f \mapsto \sum_{1 \leq i < j \leq n} f^{(i\;j)}. \]
More explicitly, it is not hard to calculate that $M(S,S) = \binom{k}{2} + \binom{n-k}{2}$, $M(S,T) = 1$ if $|S \cap T| = k-1$, and $M(S,T) = 0$ otherwise. In particular, $M$ belongs to the Bose--Mesner algebra. Lemma~\ref{lem:total-influence}, which we prove in Section~\ref{sec:influences}, shows that for $d \leq k$, the subspace spanned by $\cY_{n,d}$ is an eigenspace of $M$ corresponding to the eigenvalue $\binom{n}{2} - d(n+1-d)$. All eigenvalues are distinct (since the maximum of the parabola $d(n+1-d)$ is at $d = (n+1)/2$), completing the proof.
\end{proof}

As an immediate corollary, we deduce that $\cY_d$ is an orthogonal basis for the $d$th eigenspace in the $(n,k)$ Johnson graph and $(n,k)$ Kneser graph, as well as any other graph on $\binom{[n]}{k}$ in which the weight on an edge $(S,T)$ depends only on $|S \cap T|$. In the Johnson graph, two sets are connected if their intersection has size $k-1$, and in the Kneser graph, two sets are connected if they are disjoint.

\medskip

The slice $\binom{[n]}{k}$ can be identified as the set of cosets of $S_k \times S_{n-k}$ inside $S_n$. Bannai and Ito~\cite{BannaiIto}, following Dunkl~\cite{Dunkl78,Dunkl79}, use this approach to determine the idempotents of the Johnson association scheme from the representations of the symmetric group $S_n$. They obtain the idempotent $J_d$ from the representation corresponding to the partition $(n-d)+d$. The basis $\cY_{n,d}$ can be derived from Young's orthogonal representation corresponding to the partition $(n-d)+d$, but we do not develop this connection here; such a construction appears in Srinivasan~\cite{Srinivasan}.

\section{Influences} \label{sec:influences}

Throughout this section, we fix some arbitrary exchangeable measure. All inner products and norms are with respect to this measure.

One of the most important quantities arising in the analysis of functions on the hypercube is the influence. In this classical case, influence is defined with respect to a single coordinate. In our case, the basic quantity is the influence of a \emph{pair} of coordinates.

\begin{definition} \label{def:influence}
Let $f \in \pH_n$. For $i,j \in [n]$, define a function $f^{(i\;j)} \in \pH_n$ by $f^{(i\;j)}(x) = f(x^{(i\;j)})$, where $x^{(i\;j)}$ is obtained by switching $x_i$ and $x_j$. The \emph{influence} of the pair $(i,j)$ is
\[ \Inf_{ij}[f] = \tfrac{1}{2} \|f^{(i\;j)} - f\|^2. \]
The $m$th \emph{total influence} of the function is
\[ \Inf^m[f] = \frac{1}{m} \sum_{1 \leq i < j \leq m} \Inf_{ij}[f]. \]
When $m = n$, we call the resulting quantity the \emph{total influence}, denoted $\Inf[f]$.
\end{definition}

We start with a triangle inequality for influences (cf.~\cite[Lemma 5.4]{Wimmer} for the Boolean case, in which the constant $\tfrac{9}{2}$ can be improved to $\tfrac{3}{2}$).

\begin{lemma} \label{lem:triangle}
Let $f \in \pH_n$. For distinct $i,j,k \in [n]$ we have
\[ \Inf_{ij}[f] \leq \tfrac{9}{2}(\Inf_{ik}[f] + \Inf_{jk}[f]). \]
\end{lemma}
\begin{proof}
The Cauchy--Schwartz inequality implies that $(a+b+c)^2 \leq 3(a^2+b^2+c^2)$. Since $(i\;j) = (i\;k)(j\;k)(i\;k)$,
\begin{align*}
\Inf_{ij}[f] &= \tfrac{1}{2} \|f - f^{(i\;k)(j\;k)(i\;k)}\|^2 \\&\leq
3(\tfrac{1}{2}\|f - f^{(i\;k)}\|^2 + \tfrac{1}{2}\|f^{(i\;k)} - f^{(j\;k)(i\;k)}\|^2 + \tfrac{1}{2}\|f^{(j\;k)(i\;k)} - f^{(i\;k)(j\;k)(i\;k)}\|^2 \\&=
6\Inf_{ik}[f] + 3\Inf_{jk}[f].
\end{align*}
The lemma is obtained by averaging with the similar inequality $\Inf_{ij}[f] \leq 6\Inf_{jk}[f] + 3\Inf_{ik}[f]$ obtained using $(i\;j) = (j\;k)(i\;k)(j\;k)$.
\end{proof}

As a consequence of the triangle inequality, we can identify a set of ``important'' coordinates for functions with low total influence (cf.~\cite[Proposition 5.3]{Wimmer}).

\begin{lemma} \label{lem:important}
Let $f \in \pH_n$. For every $\tau > 0$ there exists a set $S \subseteq [n]$ of size $O(\Inf[f]/\tau)$ such that $\Inf_{ij}[f] < \tau$ whenever $i,j \notin S$.
\end{lemma}
\begin{proof}
Construct a graph $G$ on the vertex set $[n]$ in which two vertices $i,j$ are connected if $\Inf_{ij}[f] \geq \tau$. Let $M$ be a maximal matching in $G$. For each $(i,j) \in M$, Lemma~\ref{lem:triangle} shows that for all $k \neq i,j$, $\Inf_{ik}[f] + \Inf_{jk}[f] = \Omega(\tau)$. Summing over all edges in $M$ we obtain $\Inf[f] = \Omega(\tau |M|)$ and so $|M| = O(\Inf[f]/\tau)$. The endpoints of $M$ form a vertex cover $S$ in $G$ of size $2|M| = O(\Inf[f]/\tau)$. Since $S$ is a vertex cover, whenever $\Inf_{ij}[f] \geq \tau$ then $(i,j)$ is an edge and so either $i \in S$ or $j \in S$ (or both). It follows that $\Inf_{ij}[f] < \tau$ whenever $i,j \notin S$.
\end{proof}

Our goal in the rest of this section is to give a formula for $\Inf^m[f]$. Our treatment closely follows Wimmer~\cite{Wimmer}. We start with a formula for $f^{(m\;m+1)}$.

\begin{lemma} \label{lem:transposition}
Let $f \in \pH_n$. For $m \in [n-1]$,
\begin{align*}
f^{(m\;m+1)} = \sum_{\substack{B \in \cB_n\colon \\ m,m+1 \notin B \text{ or} \\ m,m+1 \in B}} \hat{f}(B) \chi_B
&+ \sum_{\substack{B \in \cB_n\colon\\ b_{i+1} = m, m+1 \notin B}} \left(\frac{1}{m-2i} \hat{f}(B) + \frac{m-2i+1}{m-2i} \hat{f}(B^{(m\;m+1)})\right) \chi_B \\
&+ \sum_{\substack{B \in \cB_n\colon\\ b_{i+1} = m+1, m \notin B}} \left(-\frac{1}{m-2i} \hat{f}(B) + \frac{m-2i-1}{m-2i} \hat{f}(B^{(m\;m+1)})\right) \chi_B.
\end{align*}
It might happen that $B \in \cB_n$ but $B^{(m\;m+1)} \notin \cB_n$, but in that case, the coefficient in front of $\hat{f}(B^{(m\;m+1)})$ vanishes.
\end{lemma}
\begin{proof}
For brevity, define $\pi = (m\;m+1)$.
We start by showing that if $B \in \cB_n$ but $B^\pi \notin \cB_n$ then the coefficient in front of $\hat{f}(B^\pi)$ vanishes. Clearly, this case can happen only if $m+1 \in B$, say $b_{i+1} = m+1$, and $m \notin B$. Since $B \in \cB_n$, $m - i = |[m+1] \setminus B| \geq i+1$. Since $B^\pi \notin \cB_n$, $m - 1 - i = |[m] \setminus B^\pi| \leq i$. We conclude that $m = 2i+1$, and indeed the corresponding coefficient is $\frac{m-2i-1}{m-2i} = 0$.

Similarly, if $b_{i+1} \in \{m,m+1\}$ for some $B \in \cB_n$ then $b_{i+1} - i - 1 = |[b_{i+1}] \setminus B| \geq i+1$, and in particular $m-2i \geq b_{i+1}-1-2i \geq 1$. This shows that the expressions appearing in the expansion of $f^\pi$ are well-defined.

Consider some $B \in \cB_n$. If $m,m+1 \notin B$ then $A < B$ iff $A^\pi < B$, and so
\[ \chi_B^\pi = \sum_{\substack{A \in \cS_n\colon\\ A < B}} \chi_{A^\pi,B} = \sum_{\substack{A \in \cS_n\colon\\ A^\pi < B}} \chi_{A,B} = \chi_B. \]
If $m,m+1 \in B$, say $b_i = m$ and $b_{i+1} = m+1$, then $\chi_{A,B}^\pi = \chi_{A',B}$, where $A'$ is obtained from $A$ by switching $a_i$ and $a_{i+1}$. Again $A < B$ iff $A' < B$, and so $\chi_B^\pi = \chi_B$ as in the preceding case.

Finally, consider some set $B \in \cB_{n,d}$ such that $B^\pi \in \cB_n$, $b_{i+1} = m$ and $m+1 \notin B$. The set $B$ doesn't necessarily belong to $\cB_n$, and in that case we define $\chi_B = 0$; under this convention, the formula $\chi_B = \sum_{A < B} \chi_{A,B}$ still holds (vacuously). We define a function $\phi$ which maps a sequence $A < B$ to a sequence $\phi(A) < B^\pi$ so that the following equation holds:
\begin{equation} \label{eq:trick}
\chi_{A,B}^\pi - \chi_{A,B}^{\vphantom{\pi}} = \chi_{\phi(A),B^\pi}.
\end{equation}
The function $\phi$ is given by
\[
\phi(A) = \begin{cases} a_1,\ldots,a_i,m,a_{i+2},\ldots,a_d & \text{if } m+1 \notin A, \\ a_1,\ldots,a_i,m,a_{i+2},\ldots,a_j,a_{i+1},a_{j+2},\ldots,a_d & \text{if } a_{j+1} = m+1. \end{cases}
\]
Since $b_{i+1} = m$ and $A < B$, in the second case necessarily $j > i$. It is not hard to verify that indeed $\phi(A) < B^\pi$.

We proceed to verify equation~\eqref{eq:trick}. Suppose first that $m+1 \notin A$. Then
\begin{align*}
\chi_{A,B}^\pi - \chi_{A,B}^{\vphantom{\pi}} &=
\prod_{k=1}^i (x_{a_k} - x_{b_k}) \times [(x_{a_{i+1}} - (m+1)) - (x_{a_{i+1}} - m)] \times \prod_{k=i+2}^d (x_{a_k} - x_{b_k}) \\ &=
\prod_{k=1}^i (x_{a_k} - x_{b_k}) \times (m - (m+1)) \times \prod_{k=i+2}^d (x_{a_k} - x_{b_k}) = \chi_{\phi(A),B^\pi}.
\end{align*}
Suppose next that $a_{j+1} = m+1$. Then
\[
\chi_{A,B}^\pi - \chi_{A,B}^{\vphantom{\pi}} = \prod_{\substack{k=1\\k\neq i+1,j+1}}^n (x_{a_k} - x_{b_k}) \times [(x_{a_{i+1}} - (m+1)) (m - x_{b_{j+1}}) - (x_{a_{i+1}} - m) (m+1 - x_{b_{j+1}})].
\]
Using $(\alpha-1)\beta - \alpha(\beta+1) = -\alpha-\beta$,
\[
(x_{a_{i+1}} - (m+1)) (m - x_{b_{j+1}}) - (x_{a_{i+1}} - m) (m+1 - x_{b_{j+1}}) = -(x_{a_{i+1}} - m) - (m - x_{b_{j+1}}) = -(x_{a_{i+1}} - x_{b_{j+1}}).
\]
Therefore
\[
\chi_{A,B}^\pi - \chi_{A,B}^{\vphantom{\pi}} = \prod_{\substack{k=1\\k\neq i+1,j+1}}^n (x_{a_k} - x_{b_k}) \times (m - (m+1)) \times (x_{a_{i+1}} - x_{b_{j+1}}) = \chi_{\phi(A),B^\pi}.
\]
This completes the proof of equation~\eqref{eq:trick}.

Every sequence in $A \in \im\phi$ satisfies $a_{i+1} = m$. Let $A < B^\pi$ be any sequence satisfying $a_{i+1} = m$. We proceed to determine $|\phi^{-1}(A)|$. For every $t \in [m-1] \setminus (A \cup B)$, the sequence $a_1,\ldots,a_i,t,a_{i+2},\ldots,a_d$ is in $\phi^{-1}(A)$. For each $j > i$ such that $a_{j+1} < m$, another sequence in $\phi^{-1}(A)$ is $a_1,\ldots,a_i,a_{j+1},a_{i+2},\ldots,a_j,m+1,a_{j+2},\ldots,a_d$. If the number of such latter indices is $J$, then $|[m-1] \setminus (A \cup B)| = m-1-2i-J$. In total, $|\phi^{-1}(A)| = m-1-2i$. We conclude that
\begin{equation} \label{eq:diff1}
\chi_B^\pi - \chi_B^{\vphantom{\pi}} = (m-2i-1) \sum_{A \in \im\phi} \chi_{A,B^\pi}.
\end{equation}

We proceed to calculate $\chi_{B^\pi} - \chi_B$. If $A < B$ then clearly $A^\pi < B^\pi$. Conversely, if $A^\pi < B^\pi$ then $A < B$ unless $a_{i+1}^\pi = m$, in which case $A^\pi \in \im\phi$. The preceding paragraph shows that every $A^\pi < B^\pi$ satisfying $a_{i+1}^\pi = m$ belongs to $\im\phi$, and so
\begin{align}
\chi_{B^\pi} - \chi_B &= \sum_{\substack{A \in \cS_{n,d}\colon \\ A < B}} (\chi_{A^\pi,B^\pi} - \chi_{A,B}) + \sum_{A \in \im\phi} \chi_{A,B^\pi} \notag \\ &=
\sum_{\substack{A \in \cS_{n,d}\colon \\ A < B}} (\chi_{A,B}^\pi - \chi_{A,B}^{\vphantom{\pi}}) + \sum_{A \in \im\phi} \chi_{A,B^\pi} =
(m-2i) \sum_{A \in \im\phi} \chi_{A,B^\pi}, \label{eq:diff2}
\end{align}
using equation~\eqref{eq:diff1}. Combining equations~\eqref{eq:diff1},\eqref{eq:diff2} together, we deduce
\[
\chi_B^\pi = \chi_B + \frac{m-2i-1}{m-2i} (\chi_{B^\pi} - \chi_B) = \frac{1}{m-2i} \chi_B + \frac{m-2i-1}{m-2i} \chi_{B^\pi}.
\]
Applying $\pi$ to both sides of equations~\eqref{eq:diff1},\eqref{eq:diff2}, we get
\[
\chi_B^{\vphantom{\pi}} - \chi_B^\pi = (m-2i-1) \sum_{A \in \im\phi} \chi_{A,B^\pi}^\pi, \quad
\chi_{B^\pi}^\pi - \chi_B^\pi = (m-2i) \sum_{A \in \im\phi} \chi_{A,B^\pi}^\pi.
\]
Therefore
\[
\chi_{B^\pi}^\pi = \chi_B^\pi + \frac{m-2i}{m-2i-1} (\chi_B^{\vphantom{\pi}} - \chi_B^\pi) = \frac{m-2i}{m-2i-1} \chi_B^{\vphantom{\pi}} - \frac{1}{m-2i-1} \chi_B^\pi =
-\frac{1}{m-2i} \chi_{B^\pi} + \frac{m-2i+1}{m-2i} \chi_B. \qedhere
\]
\end{proof}

Lemma~\ref{lem:transposition} allows us to come up with a remarkable formula for $\sum_{1\leq i < m} f^{(i\;m)}$, showing that the basis vectors $\chi_B$ are eigenvectors of this operator (cf.~\cite[Proposition 4.2]{Wimmer}, in which $\lambda_m(B)$ is the content of $m$ in the Young tableau of shape $n-|B|,|B|$ having bottom row $B$).

\begin{lemma} \label{lem:transposition-sum}
Let $f \in \pH_n$. For $m \in [n]$,
\[
\sum_{1 \leq i < m} f^{(i\;m)} = \sum_{B \in \cB_n} \lambda_m(B) \hat{f}(B) \chi_B, \quad \text{where } \lambda_m(B) = \begin{cases} i - 2 & \text{if } b_i = m, \\ m - i & \text{if } b_{i-1} < m < b_i, \end{cases}
\]
using the conventions $b_0 = -\infty$ and $b_{d+1} = \infty$ for $B \in \cB_{n,d}$.
\end{lemma}
\begin{proof}
The proof is by induction on $m$. When $m = 1$, the sum in question vanishes, and so we need to prove that $\lambda_1(B) = 0$ for all $B \in \cB_n$. Indeed, $b_1 \geq 2$ for all $B \in \cB_n$ and so $\lambda_1(B) = 0$. Suppose now that the formula holds for some $m$. Let $\pi = (m\;m+1)$. Since $(i\;m+1) = \pi (i\;m) \pi$,
\[
\sum_{1 \leq i < m+1} f^{(i\;m+1)} = \bigg(\sum_{1 \leq i < m} (f^\pi)^{(i\;m)}\bigg)^\pi + f^\pi.
\]
Lemma~\ref{lem:transposition} shows that
\begin{align*}
f^\pi = \sum_{\substack{B \in \cB_n\colon \\ m,m+1 \notin B \text{ or} \\ m,m+1 \in B}} &\hat{f}(B) \chi_B \\
+ \sum_{\substack{B^\pi \in \cB_n\colon\\ b_{i+1} = m, m+1 \notin B}} &\left(\frac{1}{m-2i} \hat{f}(B) + \frac{m-2i+1}{m-2i} \hat{f}(B^\pi)\right) \chi_B \\
+ &\left(-\frac{1}{m-2i} \hat{f}(B^\pi) + \frac{m-2i-1}{m-2i} \hat{f}(B)\right) \chi_{B^\pi}.
\end{align*}
The induction hypothesis implies that
\begin{align*}
\sum_{1 \leq i < m} (f^\pi)^{(i\;m)} = \sum_{\substack{B \in \cB_n\colon \\ m,m+1 \notin B \text{ or} \\ m,m+1 \in B}} & \lambda_m(B) \hat{f}(B) \chi_B \\
+ \sum_{\substack{B^\pi \in \cB_n\colon\\ b_{i+1} = m, m+1 \notin B}} & (i-1) \left(\frac{1}{m-2i} \hat{f}(B) + \frac{m-2i+1}{m-2i} \hat{f}(B^\pi)\right) \chi_B \\
+ &(m-i-1) \left(-\frac{1}{m-2i} \hat{f}(B^\pi) + \frac{m-2i-1}{m-2i} \hat{f}(B)\right) \chi_{B^\pi}.
\end{align*}
Another application of Lemma~\ref{lem:transposition} gives
\begin{align*}
\bigg(\sum_{1 \leq i < m} (f^\pi)^{(i\;m)}\bigg)^\pi = \sum_{\substack{B \in \cB_n\colon \\ m,m+1 \notin B \text{ or} \\ m,m+1 \in B}} &\lambda_m(B) \hat{f}(B) \chi_B \\
+ \sum_{\substack{B^\pi \in \cB_n\colon\\ b_{i+1} = m, m+1 \notin B}} &\left[\frac{(m-2i)(m-i-1)-1}{m-2i} \hat{f}(B) - \frac{m-2i+1}{m-2i} \hat{f}(B^\pi)\right] \chi_B \\
+ &\left[-\frac{m-2i-1}{m-2i} \hat{f}(B) + \frac{(m-2i)(i-1)+1}{m-2i} \hat{f}(B^\pi)\right] \chi_{B^\pi}.
\end{align*}
We conclude that
\begin{align*}
\sum_{1 \leq i < m+1} f^{(i\;m+1)} &= \bigg(\sum_{1 \leq i < m} (f^\pi)^{(i\;m)}\bigg)^\pi + f^\pi \\ &=
\sum_{\substack{B \in \cB_n\colon \\ m,m+1 \notin B \text{ or} \\ m,m+1 \in B}} (\lambda_m(B)+1) \hat{f}(B) \chi_B
+ \sum_{\substack{B^\pi \in \cB_n\colon\\ b_{i+1} = m, m+1 \notin B}} (m-i-1) \hat{f}(B) \chi_B + (i-1) \hat{f}(B^\pi) \chi_{B^\pi}.
\end{align*}
It remains to verify that these coefficients match $\lambda_{m+1}(B)$. If $m,m+1 \notin B$ then $\lambda_m(B) = m-i$, where $b_{i-1} < m < b_i$. Since $m+1 \notin B$, also $b_{i-1} < m+1 < b_i$, and so $\lambda_{m+1}(B) = m+1-i = \lambda_m(B) + 1$. If $m,m+1 \in B$, say $b_i = m$ and $b_{i+1} = m+1$, then $\lambda_{m+1}(B) = (i+1)-2 = (i-2)+1 = \lambda_m(B) + 1$. Finally, suppose that $b_{i+1} = m$ and $m+1 \notin B$. In this case, $b_{i+1} < m+1 < b_{i+2}$ and so $\lambda_{m+1}(B) = (m+1)-(i+2) = m-i-1$, while $b^\pi_{i+1} = m+1$ and so $\lambda_{m+1}(B^\pi) = (i+1)-2 = i-1$. All cases match the coefficients in the displayed formula, completing the proof.
\end{proof}

Lemma~\ref{lem:transposition-sum} shows that the elements of $\cY_n$ are eigenvectors of the operators $f \mapsto \sum_{1 \leq i < m} f^{(i\;m)}$ for all $m$. In the terminology of Vershik and Okounkov~\cite{VO}, this makes $\cY_n$ a Gelfand--Tsetlin basis. Srinivasan~\cite{Srinivasan} shows that the Gelfand--Tsetlin basis is unique in our setting, and so the basis he constructs, which is also a Gelfand--Tsetlin basis, is the same as $\cY_n$.

Lemma~\ref{lem:transposition-sum} allows us to give a formula for the $m$th total influence of a function (cf.~\cite[Definition 4.3]{Wimmer}).

\begin{lemma} \label{lem:total-influence}
Let $f \in \pH_n$. For $m \in [n]$,
\[
\sum_{1 \leq i < j \leq m} f^{(i\;j)} = \sum_{B \in \cB_n} \tau_m(B) \chi_B, \quad \text{where } \tau_m(B) = \frac{m(m-1)}{2} - |B \cap [m]|(m+1 - |B \cap [m]|).
\]
Moreover,
\[ \Inf^m[f] = \sum_{B \in \cB_n} \frac{|B \cap [m]|(m+1 - |B \cap [m]|)}{m} \hat{f}(B)^2 c_B \|\chi_{|B|}\|^2. \]
In particular,
\[ \Inf[f] = \sum_{B \in \cB_n} \frac{|B|(n+1-|B|)}{n} \hat{f}(B)^2 c_B \|\chi_{|B|}\|^2. \]
\end{lemma}
\begin{proof}
Lemma~\ref{lem:transposition-sum} proves the first formula with $\tau_m(B) = \sum_{k=2}^m \lambda_k(B)$. We can write $\lambda_k(B) = \lambda'_k(B) + \lambda''_k(B)$, where $\lambda'_k(B)$ is non-zero when $k \in B$, and $\lambda''_k(B)$ is non-zero when $k \notin B$. It follows from the definition that
\[ \lambda'_{b_1}(B),\lambda'_{b_2}(B),\ldots,\lambda'_{b_{|B|}}(B) = -1,0,\ldots,|B|-2. \]
In particular,
\[ \sum_{i=1}^m \lambda'_i(B) = -1 + \cdots + (|B \cap [m]|-2) = \frac{|B \cap [m]| (|B \cap [m]|-3)}{2}. \]
Consider the sequence $\lambda''_1(B),\ldots,\lambda''_n(B)$, in which we omit $\lambda''_k(B)$ for all $k \in B$. The sequence starts at $\lambda''_1(B) = 1-1 = 0$. Let $\lambda''_k(B) = k-i$, where $b_{i-1} < k < b_i$. If $b_i \neq k+1$ then $\lambda''_{k+1}(B) = (k+1)-i = \lambda''_k(B)+1$. If $b_i = k+1,\ldots,b_{i+\ell-1} = k+\ell$ and $k+\ell+1 \leq n$ then $\lambda''_{k+1}(B) = \cdots = \lambda''_{k+\ell} = 0$ and $\lambda''_{k+\ell+1}(B) = (k+\ell+1)-(i+\ell) = \lambda''_k(B)+1$. In other words,
\[ \lambda''_1(B),\ldots,\lambda''_n(B) = 0,1,\ldots,n-|B|-1. \]
In particular,
\[ \sum_{i=1}^m \lambda''_i(B) = 0 + \cdots + (m-|B \cap [m]|-1) = \frac{(m - |B \cap [m]|) (m - |B \cap [m]| - 1)}{2}. \]
In total,
\begin{align*}
\tau_m(B) &= \frac{|B \cap [m]| (|B \cap [m]| - 3) + (m - |B \cap [m]|) (m - |B \cap [m]| - 1)}{2} \\ &= \frac{m(m-1)}{2}  - |B \cap [m]|(m+1 - |B \cap [m]|),
\end{align*}
completing the proof of the first formula.

In order to compute the $m$th total influence, notice that
\[ \Inf_{ij}[f] = \tfrac{1}{2}\|f^{(i\;j)}-f\|^2 = \tfrac{1}{2}\|f\|^2 + \tfrac{1}{2}\|f^{(i\;j)}\|^2 - \langle f,f^{(i\;j)} \rangle = \|f\|^2 - \langle f,f^{(i\;j)} \rangle, \]
using the exchangeability of the measure. Therefore
\[
\Inf^m[f] = \frac{1}{m} \left(\frac{m(m-1)}{2} \|f\|^2 - \left\langle f,\sum_{1 \leq i < j \leq m} f^{(i\;j)} \right\rangle\right).
\]
The formula for $\Inf^m[f]$ now immediately follows from the orthogonality of the basis (Theorem~\ref{thm:orthogonal}) and the norms stated in Theorem~\ref{thm:norms}.
\end{proof}

As a simple corollary, we obtain a version of Poincar\'e's inequality.

\begin{lemma} \label{lem:poincare}
For any $f \in \pH_{n,d}$ we have
\[ \VV[f] \leq \Inf[f] \leq d \VV[f]. \]
\end{lemma}
\begin{proof}
Lemma~\ref{lem:moments} and Lemma~\ref{lem:total-influence} give the formulas
\begin{align*}
\VV[f] &= \sum_{\substack{B \in \cB_n\colon \\ B \neq \emptyset}} \hat{f}(B)^2 c_B \|\chi_{2,\ldots,2|B|}\|^2, \\
\Inf[f] &= \sum_{\substack{B \in \cB_n\colon \\ B \neq \emptyset}} \frac{|B|(n+1-|B|)}{n} \hat{f}(B)^2 c_B \|\chi_{2,\ldots,2|B|}\|^2.
\end{align*}
The left inequality follows from $|B|(n+1-|B|) \geq n$, and the right inequality from $n+1-|B| \leq n$.
\end{proof}

\section{Wimmer--Friedgut theorem} \label{sec:wimmer-friedgut}

For a Boolean function $f$ on the Boolean hypercube $\{0,1\}^n$, the influence of variable $i$ is
\[ \Inf_i[f] = \Pr_{x_1,\ldots,x_n \in \{0,1\}^n} [f(x_1,\ldots,x_n) \neq f(x_1,\ldots,1-x_i,\ldots,x_n)], \]
under the uniform distribution over the hypercube. The total influence of the function is $\Inf[f] = \sum_{i=1}^n \Inf_i[f]$.
Friedgut~\cite{FriedgutJunta} proved that for every $\epsilon > 0$, every Boolean function $f$ is $\epsilon$-close to a function $g$ depending on $2^{O(\Inf[f]/\epsilon)}$ coordinates, that is, $\Pr_{x \in \{0,1\}^n}[f(x) \neq g(x)] \leq \epsilon$.

Friedgut's theorem can be seen as an ``inverse theorem'' corresponding to the easy fact that a Boolean function depending on $d$ coordinates has total influence at most $d$. Since the total influence is bounded by the degree, Friedgut's theorem can also be seen as a strengthening of the result of Nisan and Szegedy~\cite{NisanSzegedy} that a Boolean function of degree $d$ depends on at most $d2^{d-1}$ coordinates.

Wimmer~\cite{Wimmer} proved an analog of Friedgut's theorem for functions on a slice of the Boolean cube. His proof takes place mostly on the symmetric group, and uses properties of Young's orthogonal representation. We rephrase his proof in terms of Young's orthogonal basis for the slice.

The proof relies crucially on a hypercontractivity property due to Lee and Yau~\cite{LeeYau}.
Before stating the property, we need to define the noise operator.

\begin{definition} \label{def:noise}
The \emph{Laplacian} operator on functions $f \in \pH_n$ is given by
\[ Lf = f - \frac{1}{\binom{n}{2}} \sum_{1 \leq i < j \leq n} f^{(i\;j)}. \]
The \emph{noise operator} $H_t$ is given by $H_t = e^{-tL}$.
\end{definition}

The Laplacian corresponds to the Markov chain applying a random transposition $(i\;j)$. Moreover, $L = I-K$ where $K$ is the transition matrix of the Markov chain.
We can expand the noise operator as
\[
H_t = e^{t(K-I)} = e^{-t} \sum_{\ell=0}^\infty \frac{t^\ell}{\ell!} K^\ell.
\]
In words, $H_t$ corresponds to applying $P(t)$ many random transpositions $(i\;j)$, where $P(t)$ is the Poisson distribution with mean $t$.

Lemma~\ref{lem:total-influence} gives a formula for $Lf$ and $H_tf$.

\begin{lemma} \label{lem:noise}
Let $f \in \pH_n$. For every $t$,
\[ Lf = \sum_{B \in \cB_n} \frac{2|B|(n+1-|B|)}{n(n-1)} \hat{f}(B) \chi_B, \quad H_tf = \sum_{B \in \cB_n} \exp\left(-t \frac{2|B|(n+1-|B|)}{n(n-1)}\right) \hat{f}(B) \chi_B. \]
\end{lemma}

The hypercontractivity result of Lee and Yau~\cite{LeeYau} gives for all $p < q$ a value of $t$ such that $\|H_t f\|_q \leq \|f\|_p$.

\begin{proposition} \label{pro:lee-yau}
Let $n,k$ be integers such that $1 \leq k \leq n-1$. The log-Sobolev constant $\rho$ of the Markov chain corresponding to the Laplacian $L$ is given by
\[ \rho^{-1} = \Theta\left(n\log \frac{n^2}{k(n-k)}\right). \]
Consequently, for every $t \geq 0$ and $1 \leq p \leq q \leq \infty$ satisfying $\frac{q-1}{p-1} \leq \exp(2\rho t)$ and all $f \in \pH_n$, $\|H_t f\|_q \leq \|f\|_p$.
\end{proposition}
\begin{proof}
The first result is~\cite[Theorem 5]{LeeYau}. Their parameter $t$ is scaled by a fraction of $n$. Furthermore, their log-Sobolev constant is the reciprocal of ours. The second result is due to Gross~\cite{Gross}, and is quoted from~\cite[Theorem 2.4]{ODonnellWimmer}.
\end{proof}

We can now state and prove the Wimmer--Friedgut theorem~\cite[Theorem 1.3]{Wimmer}.

\begin{definition} \label{def:depends}
A function $f \in \pH_n$ \emph{depends (only) on a set $S \subseteq [n]$} if $f = f^{(i\;j)}$ whenever $i,j \notin S$.
\end{definition}

\begin{theorem} \label{thm:wimmer}
Let $n,k$ be integers such that $1 \leq k \leq n/2$, and define $p = k/n$. For every Boolean function $f$ on the slice and any $\epsilon > 0$ there exists a Boolean function $g$ depending on $O(p^{-O(\Inf[f]/\epsilon)})$ coordinates such that $\Pr[f \neq g] \leq \epsilon$.
\end{theorem}
\begin{proof}
Let $\tau > 0$ be a parameter to be determined. Lemma~\ref{lem:important} shows that there exists a set $S \subseteq [n]$ of size $m = O(\Inf[f]/\tau)$ such that $\Inf_{ij}[f] < \tau$ whenever $i,j \notin S$. Without loss of generality, we can assume that $S = \{n-m+1,\ldots,n\}$. Let $h$ be the function obtained from $f$ by averaging over all permutations of $[n-m]$, and let $g$ be the Boolean function obtained from rounding $h$ to $\{0,1\}$. Note that $h$ and $g$ both depend only on the last $m$ coordinates. Lemma~\ref{lem:averaging} shows that the Young--Fourier expansion of $h$ is obtained from that of $f$ by dropping all terms $\hat{f}(B) \chi_B$ for which $B \cap [n-m] \neq \emptyset$. Therefore Lemma~\ref{lem:moments} shows that
\[ \Pr[f \neq g] = \|f-g\|^2 \leq 2\|f-h\|^2 = 2\sum_{\substack{B \in \cB_n\colon \\ B \cap [n-m] \neq \emptyset}} \hat{f}(B)^2 c_B \|\chi_{|B|}\|^2. \]

We bound the sum on the right-hand side by considering separately large sets and small sets. Let $d \leq n/2$ be a parameter to be determined. For the large sets, we have
\[ \sum_{\substack{B \in \cB_n\colon \\ |B| \geq d}} \hat{f}(B)^2 c_B \|\chi_{|B|}\|^2 \leq \frac{n}{d(n+1-d)} \sum_{\substack{B \in \cB_n\colon \\ |B| \geq d}} \frac{|B|(n+1-|B|)}{n} c_B \|\chi_{|B|}\|^2 = \frac{n}{d(n+1-d)} \Inf[f], \]
using Lemma~\ref{lem:total-influence}.

In order to bound the part corresponding to small sets, we need to apply hypercontractivity.
Choose $(p,q) = (\tfrac{4}{3},2)$ in Proposition~\ref{pro:lee-yau} to deduce that for $t = \Theta(n\log \frac{n^2}{k(n-k)}) = \Theta(n\log p^{-1})$ we have $\|H_tF\|_2 \leq \|F\|_{4/3} = \|F\|_2^{3/2}$ for every function $F$ on the slice which takes the values $0,\pm 1$. In particular, choosing $F = f - f^{(i\;j)}$ we obtain $(2\Inf_{ij}[H_tf])^{1/2} \leq (2\Inf_{ij}[f])^{3/4}$ and so $\Inf_{ij}[H_t f] \leq \sqrt{2} \Inf_{ij}[f]^{3/2}$. Therefore
\[
\Inf^{n-m}[H_t f] \leq \frac{\sqrt{2}}{n-m} \sum_{1 \leq i < j \leq n-m} \Inf_{ij}[f]^{3/2} \leq \frac{\sqrt{2}n}{n-m} \sqrt{\tau} \Inf[f].
\]
On the other hand, Lemma~\ref{lem:total-influence} and Lemma~\ref{lem:noise} show that
\begin{align*}
\Inf^{n-m}[H_t f] &=
\sum_{B \in \cB_n} \exp \left(-t\frac{2|B|(n+1-|B|)}{n(n-1)}\right) \frac{|B \cap [n-m]|(n-m+1-|B \cap [n-m]|)}{n-m} \hat{f}(B)^2 c_B \|\chi_{|B|}\|^2 \\ &\geq
\sum_{\substack{B \in \cB_n\colon \\ |B| < d \\ B \cap [n-m] \neq \emptyset}} \exp \left(-t\frac{2d(n+1-d)}{n(n-1)}\right) \hat{f}(B)^2 c_B \|\chi_{|B|}\|^2 \\ &=
\sum_{\substack{B \in \cB_n\colon \\ |B| < d \\ B \cap [n-m] \neq \emptyset}} p^{\Theta(d)} \hat{f}(B)^2 c_B \|\chi_{|B|}\|^2.
\end{align*}
Altogether,
\[
\sum_{\substack{B \in \cB_n\colon \\ |B| < d \\ B \cap [n-m] \neq \emptyset}} \hat{f}(B)^2 c_B \|\chi_{|B|}\|^2 \leq \sqrt{2} p^{-\Theta(d)} \frac{n}{n-m} \sqrt{\tau} \Inf[f].
\]

Putting both bounds together, we deduce
\[
\Pr[f \neq g] \leq 2\frac{n}{d(n+1-d)} \Inf[f] + 2\sqrt{2} p^{-\Theta(d)} \frac{n}{n-m} \sqrt{\tau} \Inf[f].
\]
We now choose $d = \Inf[f]/(8\epsilon)$ and $\tau = p^{Cd}$ for an appropriate constant $C>0$, so that $m = O(\Inf[f] p^{-O(\Inf[f]/\epsilon)}) = O(p^{-O(\Inf[f]/\epsilon)})$. We can assume that $d,m \leq n/2$, since otherwise the theorem is trivial. We can choose $C$ so that
\[
\Pr[f \neq g] \leq \frac{\epsilon}{2} + O(p^{Cd/3} \Inf[f]) = \frac{\epsilon}{2} + O(2^{-Cd/3} \Inf[f]) = \frac{\epsilon}{2} + O(2^{-Cd/3} d) \frac{\epsilon}{2}.
\]
For an appropriate choice of $C$, the second term is at most $\epsilon/2$, completing the proof.
\end{proof}

\bibliographystyle{plain}
\bibliography{WimmerFriedgut}

\begin{thebibliography}{10}

\bibitem{BannaiIto}
Eiichi Bannai and Tatsuro Ito.
\newblock {\em Algebraic Combinatorics {I}: {A}ssociation schemes}.
\newblock Benjamin/Cummings Pub. Co., 1984.

\bibitem{Dunkl76}
Charles~F. Dunkl.
\newblock A {K}rawtchouk polynomial addition theorem and wreath products of
  symmetric groups.
\newblock {\em Indiana Univ. Math. J.}, 25:335--358, 1976.

\bibitem{Dunkl78}
Charles~F. Dunkl.
\newblock An addition theorem for {H}ahn polynomials: the spherical functions.
\newblock {\em SIAM J. MATH. Anal.}, 9:627--637, 1978.

\bibitem{Dunkl79}
Charles~F. Dunkl.
\newblock Orthogonal functions on some permutation groups.
\newblock In {\em Relations between combinatorics and other parts of
  mathematics}, volume~34 of {\em Proc. Sump. Pure Math.}, pages 129--147,
  Providence, RI, 1979. Amer. Math. Soc.

\bibitem{Filmus}
Yuval Filmus.
\newblock Friedgut--{K}alai--{N}aor theorem for slices of the {B}oolean cube,
  2014.
\newblock Submitted.

\bibitem{FKMW}
Yuval Filmus, Guy Kindler, Elchanan Mossel, and Karl Wimmer.
\newblock Invariance principle on the slice, 2015.
\newblock Submitted.

\bibitem{FM}
Yuval Filmus and Elchanan Mossel.
\newblock Harmonicity and invariance on slices of the {B}oolean cube, 2015.
\newblock Submitted.

\bibitem{FranklGraham}
P{\'e}ter Frankl and Ron~L. Graham.
\newblock Old and new proofs of the {E}rd{\H{o}}s-{K}o-{R}ado theorem.
\newblock {\em J. of Sichuan Univ. Natural Science Edition}, 26:112--122, 1989.

\bibitem{FriedgutJunta}
Ehud Friedgut.
\newblock Boolean functions with low average sensitivity depend on few
  coordinates.
\newblock {\em Combinatorica}, 18(1):27--36, 1998.

\bibitem{FKN}
Ehud Friedgut, Gil Kalai, and Assaf Naor.
\newblock Boolean functions whose {F}ourier transform is concentrated on the
  first two levels.
\newblock {\em Adv. Appl. Math.}, 29(3):427--437, 2002.

\bibitem{Gross}
Leonard Gross.
\newblock Logarithmic {S}obolev inequalities.
\newblock {\em Amer. J. Math.}, 97:1061--1083, 1975.

\bibitem{KKL}
Jeff Kahn, Gil Kalai, and Nati Linial.
\newblock The influence of variables on {B}oolean functions.
\newblock In {\em Proceedings of the 29th Symposium on the Foundations of
  Computer Science}, pages 68--80, White Plains, 1988.

\bibitem{Kindler}
Guy Kindler.
\newblock {\em Property testing, {PCP} and Juntas}.
\newblock PhD thesis, Tel-Aviv University, 2002.

\bibitem{KindlerSafra}
Guy Kindler and Shmuel Safra.
\newblock Noise-resistant {B}oolean functions are juntas, 2004.
\newblock Manuscript.

\bibitem{LeeYau}
Tzong-Yau Lee and Horng-Tzer Yau.
\newblock Logarithmic {S}obolev inequality for some models of random walks.
\newblock {\em Ann. Prob.}, 26(4):1855--1873, 1998.

\bibitem{Lovasz}
L\'aszl\'o Lov\'asz.
\newblock On the {S}hannon capacity of a graph.
\newblock {\em IEEE Trans. Inform. Theory}, 25:1--7, 1979.

\bibitem{MOO}
Elchanan Mossel, Ryan O'Donnell, and Krzysztof Oleszkiewicz.
\newblock Noise stability of functions with low influences: Invariance and
  optimality.
\newblock {\em Ann. Math.}, 171:295--341, 2010.

\bibitem{NisanSzegedy}
Noam Nisan and Mario Szegedy.
\newblock On the degree of {B}oolean functions as real polynomials.
\newblock {\em Comp. Comp.}, 4(4):301--313, 1994.

\bibitem{ODonnell}
Ryan O'Donnell.
\newblock {\em Analysis of {B}oolean Functions}.
\newblock Cambridge University Press, 2014.

\bibitem{ODonnellWimmer}
Ryan O'Donnell and Karl Wimmer.
\newblock {KKL}, {K}ruskal--{K}atona, and monotone nets.
\newblock In {\em 50th Annual Symposium on Foundations of Computer Science
  (FOCS 2009)}, pages 725--734, 2009.

\bibitem{OW1}
Ryan O'Donnell and Karl Wimmer.
\newblock {KKL}, {K}ruskal--{K}atona, and monotone nets.
\newblock {\em SIAM J. Comput.}, 42(6):2375--2399, 2013.

\bibitem{OW2}
Ryan O'Donnell and Karl Wimmer.
\newblock Sharpness of {KKL} on {S}chreier graphs.
\newblock {\em Elec. Comm. Prob.}, 18:12:1--12, 2013.

\bibitem{Srinivasan}
Murali~K. Srinivasan.
\newblock Symmetric chains, {G}elfand--{T}setlin chains, and the {T}erwilliger
  algebra of the binary {H}amming scheme.
\newblock {\em J. Algebr. Comb.}, 34(2):301--322, 2011.

\bibitem{VO}
Anatoly Vershik and Andrei Okounkov.
\newblock A new approach to the representation theory of the symmetric groups
  --- {II}.
\newblock {\em J. Math. Sci.}, 131:5471--5494, 2005.

\bibitem{Wilson}
Richard~M. Wilson.
\newblock The exact bound in the {E}rd{\H{o}}s-{K}o-{R}ado theorem.
\newblock {\em Combinatorica}, 4:247--257, 1984.

\bibitem{Wimmer}
Karl Wimmer.
\newblock Low influence functions over slices of the {B}oolean hypercube depend
  on few coordinates.
\newblock In {\em Conference on Computational Complexity (CCC 2014)}, pages
  120--131, 2014.

\end{thebibliography}
\end{document}